\documentclass[preprint,12pt]{elsarticle}
\usepackage{amsthm,amsmath,amssymb}
\usepackage[colorlinks=true,citecolor=black,linkcolor=black,urlcolor=blue]{hyperref}
\usepackage{graphicx}
\usepackage{float}
\usepackage{epstopdf}
\usepackage{epsfig}
\usepackage{extarrows,chngpage,array,float,eqnarray}     %长等号的宏包

\theoremstyle{plain}
\newtheorem{theorem}{Theorem}
\newtheorem{lemma}[theorem]{Lemma}

\theoremstyle{definition}
\newtheorem{definition}[theorem]{Definition}

\theoremstyle{remark}

\title{\textbf{A direct and elementary proof of the well-definedness of the interior and exterior polynomials of hypergraphs}}
\author{
	\small Xiaxia Guan, \ \ Xian'an Jin\footnote{Corresponding author}, \ \ Tianlong Ma\\[0.2cm]
	\small School of Mathematical Sciences, Xiamen University,\\
	\small Xiamen, Fujian 361005, China\\[0.2cm]
	\small E-mails: gxx0544@126.com; xajin@xmu.edu.cn; tianlongma@aliyun.com}
\date{}
\begin{document}
\begin{abstract}
T. K\'{a}lm\'{a}n (A version of Tutte's polynomial for hypergraphs, Adv. Math. 244 (2013) 823-873.) introduced the interior and exterior polynomials which are generalizations of the Tutte polynomial $T(x,y)$ on plane points $(1/x,1)$ and $(1,1/y)$ to hypergraphs. The two polynomials are defined under a fixed ordering of hyperedges, and are proved to be independent of the ordering using techniques of polytopes. In this paper, similar to the Tutte's original proof we provide a direct and elementary proof for the well-definedness of the interior and exterior polynomials of hypergraphs.
\vskip0.2cm

\vskip0.2cm

\end{abstract}
\maketitle
\section{Introduction}
\noindent

It is well-known that the Tutte polynomial $T(x, y)$ \cite{Tutte} is an important invariant of graphs.  Motivated by the study of the HOMFLY polynomial \cite{Freyd,Hom2}, which is an important invariant in knot theory,  K\'{a}lm\'{a}n, in \cite{Kalman1}, introduced  the interior polynomial $I_{\mathcal{H}}(x)$ and the exterior polynomial $X_{\mathcal{H}}(y)$ which are generalization of $T(1/x,1)$ and $T(1,1/y)$ to hypergraphs $\mathcal{H}$, respectively.

A \emph{hypergraph} is a  pair $\mathcal{H}=(V,E)$, where $V$ is a finite set and $E$ is a finite multiset of non-empty subsets of $V$. Elements of $V$ are called \emph{vertices} and  elements of $E$ are called \emph{hyperedges}. For a hypergraph $\mathcal{H}=(V,E)$,  its associated bipartite graph $Bip \mathcal{H}$ is a bipartite graph which the sets $V$ and $E$ are the colour classes of  $Bip \mathcal{H}$, and an element $v$ of $V$ is connected to an element $e$ of $E$ if and only if $v\in e$. In this paper, we consider a connected hypergraph  $\mathcal{H}=(V,E)$, that is,  $Bip \mathcal{H}$ is connected.

A hypertree in a connected hypergraph $\mathcal{H}$ is a function $f$: $E\rightarrow N=\{0,1,2,\cdots\}$ such that a spanning tree $\tau$ of its associated bipartite graph $Bip \mathcal{H}$ can be found with the degree of $e$ in $\tau$ $d_{\tau}(e)=f(e)+1$ for any $e\in E$. We call that $\tau$ \emph{realises} or \emph{induces} $f$. We denote the set of all hypertrees in $\mathcal{H}$ with $B_{\mathcal{H}}$.
Let $f$ be a hypertree and $e$, $e'$ be two distinct hyperedges of $\mathcal{H}$. We say that $f$ is the hypertree  such that \emph{a transfer of valence} is possible from  $e$ to $e'$ if  the function $f'$ obtained from $f$ by decreasing $f(e)$ by 1 and increasing $f(e')$ by 1 is also a hypertree.

Given an order on $E$, let $f$ be a hypertree. A hyperedge $e\in E$ is \emph{internally active} with respect to the hypertree $f$  if for any hyperedge $e'<e$, $e'$  can not be transferred  valence from $e$. We say that with respect to  $f$, a hyperedge $e\in E$ is \emph{internally inactive} if it is not internally active. Let $\iota(f)$ ($\overline{\iota}(f)$, resp.) denote the number of internally  active (inactive, resp.) hyperedges with respect to $f$. This value is called the \emph{internal activity} (\emph{internal inactivity}, resp.) of $f$.
A hyperedge $e\in E$ is \emph{externally active} with respect to  $f$  if for any hyperedge $e'<e$, $e'$  can not transfer  valence to $e$. We say that with respect to  $f$, a hyperedge $e\in E$ is \emph{externally inactive} if it is not externally active. Let $\epsilon(f)$ ($\overline{\epsilon}(f)$, resp.) denote the number of externally  active (inactive, resp.) hyperedges with respect to $f$. This value is called the \emph{external activity} (\emph{external inactivity}, resp.) of $f$.
We denote the \emph{interior polynomial} $I_{\mathcal{H}}(x)=\sum\limits_{f\in B_{\mathcal{H}}} x^{\overline{\iota}(f)}$ and the \emph{exterior polynomial}  $X_{\mathcal{H}}(y)=\sum\limits_{f\in B_{\mathcal{H}}} y^{\overline{\epsilon}(f)}$.

In \cite{Kalman1} and \cite{Kalman3}, K\'{a}lm\'{a}n and Postnikov proved that
the interior polynomial and the exterior polynomial are two invariants of hypergraphs by a straightforward argument using technique of polytope and indirect approach by counting Ehrhart-type lattice point, respectively, which can be stated as follows.
\begin{theorem}[\cite{Kalman1}]\label{thm 4}
 The interior and exterior polynomials of a connected hypergraph $\mathcal{H}=(V,E)$ do not depend on the chosen order on $E$.
\end{theorem}
In this paper, similar to the idea that Tutte proves the independence of the ordering of edges of his polynomial in \cite{Tutte}, we provide a direct proof of Theorem \ref{thm 4}. Our proof is also elementary, without the use of polytopes.

\section{Preliminaries}
\noindent

In this section, to prove our Theorem \ref{thm 4}, we first introduce several definitions and significant conclusions. Moreover, we shall list some known results and prove them by the basic technique.
\begin{definition}
Let $G=(E\cup V, \varepsilon)$ be a connected bipartite graph. For a subset $E'\subset E$, let $G|_{E'}$ denote the bipartite graph formed by $E'$, all edges of $G$ incident with elements of $E'$ and their endpoints in $V$.
 We denote $\mu(E')=0$ for $E'=\emptyset$, and $\mu(E')=|\bigcup E'|-c(E')$ for $E'\neq \emptyset$, where $c(E')$ is the number of connected components of $G|_{E'}$ and $\bigcup E'=V\cap (G|_{E'})$.
\end{definition}
In \cite{Kalman1}, the sufficient and necessary conditions of hypertrees of a hypergraph was obtained.
\begin{theorem} [\cite{Kalman1}] \label{thm 1}
Let $\mathcal{H}=(V,E)$ be a connected hypergraph and $Bip \mathcal{H}$ be the bipartite graph associated to the hypergraph $\mathcal{H}$. Let $f$ be a hypertree of $\mathcal{H}$. Then

(1) $0\leq f(e)\leq d_{Bip \mathcal{H}}(e)-1$ for all $e\in E$;

(2) $\sum\limits_{e\in E}f(e)=|V|-1$;

(3) $\sum\limits_{e\in E'}f(e)\leq \mu(E')$ for all $ E'\subset E$.
\end{theorem}

\begin{lemma}[\cite{Kalman1}]
Let $\mathcal{H}=(V,E)$ be a connected hypergraph and $Bip \mathcal{H}$ be the bipartite graph associated to  $\mathcal{H}$. If $g:E\rightarrow N$ is a function satisfying two conditions as follows: (i) $g(e)\geq0$ for any hyperedge $e\in E$; (ii)  $\sum\limits_{e\in E'}g(e)\leq \mu(E')$ for any nonempty subset $E'\subset E$ but not necessarily $\sum\limits_{e\in E}f(e)=|V|-1$, then there exists a cycle-free subgraph $G'$ of $Bip \mathcal{H}$ with $d_{G'}(e)=f(e)+1$.
\end{lemma}
It is clear that $G'$ is a spanning tree of $Bip \mathcal{H}$ if $\sum\limits_{e\in E}f(e)=|V|-1$ holds. Then we have the following conclusion.
\begin{lemma}\label{lem 1}
If $g:E\rightarrow N$ is a function satisfying three conditions as follows: (i) $g(e)\geq0$ for any hyperedge $e\in E$; (ii) $\sum\limits_{e\in E}g(e)=|V|-1$; (iii) $\sum\limits_{e\in E'}g(e)\leq \mu(E')$ for any nonempty subset $ E^{\prime}\subset E$, then $g$ is a hypertree.
\end{lemma}
For any nonempty subset $ E^{\prime}\subset E$, if there is  a hypertree $f_{i}$ such that $\sum\limits_{e\in E^{\prime}}g(e)\leq \sum\limits_{e\in E^{\prime}}f_{i}(e)$, then $\sum\limits_{e\in E^{\prime}}g(e)\leq \mu(E')$ by Theorem \ref{thm 1} (3). Thus we have the following result.
\begin{lemma}[\cite{Kalman1}]\label{lem 2}
Let $\mathcal{H}=(V,E)$ be a connected hypergraph, and \{$f_{i}$\} be the set of hypertrees of $\mathcal{H}$.
Let $g:E\rightarrow N$ is a function satisfying two conditions as follows: (i) $g(e)\geq0$ for any hyperedge $e\in E$; (ii) $\sum\limits_{e\in E}g(e)=|V|-1$. If for any nonempty subset $ E^{\prime}\subset E$, there is  a hypertree $f_{i}$ such that $\sum\limits_{e\in E^{\prime}}g(e)\leq \sum\limits_{e\in E'}f_{i}(e)$ holds, then $g$ is a hypertree.
\end{lemma}
One  sufficient condition that a hyperedge can transfer valence to another  hyperedge for a hypertree, which is proved by technique of the hypertree polytope in \cite{Kalman1}, is given as follows. However, we prove it again based on Lemma \ref{lem 2}.
\begin{lemma}[\cite{Kalman1}]\label{lem 3}
	Let $\mathcal{H}=(V,E)$ be a connected hypergraph, and $f$ be a hypertree of $\mathcal{H}$. Let $e_{i}$ $(i=1,2,3)$ be distinct hyperedges of $\mathcal{H}$. With respect to  $f$, if  $e_{1}$ can transfer valence to $e_{2}$, and $e_{2}$ can transfer valence to $e_{3}$ with respect to $f$,  then  $e_{1}$ can transfer valence to $e_{3}$ with respect to $f$.
\end{lemma}
\begin{proof}
	Let the hypertree $f$ be ($\cdots,f(e_{3}),\cdots,f(e_{2}),\cdots,f(e_{1}),\cdots$). Since with respect to  $f$, $e_{1}$ can transfer valence to $e_{2}$, and $e_{2}$ can transfer valence to $e_{3}$,  there are the following two hypertrees.
	
	$f_{1}$: ($\cdots,f(e_{3}),\cdots,f(e_{2})+1,\cdots,f(e_{1})-1,\cdots$);
	
	$f_{2}$: ($\cdots,f(e_{3})+1,\cdots,f(e_{2})-1,\cdots,f(e_{1}),\cdots$).
	
	Next we prove that $f_{3}$: ($\cdots,f(e_{3})+1,\cdots,f(e_{2}),\cdots,f(e_{1})-1,\cdots$) is a hypertree of $\mathcal{H}$.
	
	It is clear that $f_{3}$ satisfies $f_{3}(e)\geq0$ for any hyperedge $e\in E$ and $\sum\limits_{e\in E}f_{3}(e)=|V|-1$. For any nonempty subset $E^{\prime}\in E$, if $e_{1}\in E'$ or $e_{1},e_{2}\notin E'$, then $\sum\limits_{e\in E'}f_{3}(e)\leq\sum\limits_{e\in E'}f_{2}(e)$. If $e_{1}\notin E'$ and $e_{2}\in E'$, then $\sum\limits_{e\in E'}f_{3}(e)\leq\sum\limits_{e\in E'}f_{1}(e)$.
	Thus, by Lemma \ref{lem 2}, $f_{3}$ is a hypertree of $\mathcal{H}$. It implies that $e_{1}$ can transfer valence to $e_{3}$  with respect to  $f$.
\end{proof}

Recall that for a hypertree $f$, by Theorem \ref{thm 1}, we have $\sum\limits_{e\in E'}f(e)\leq \mu(E')$ for all $E'\subset E$. In particular, we give the following definition when $\sum\limits_{e\in E'}f(e)= \mu(E')$ for a subset $E'\subset E$.
\begin{definition}
	Let $\mathcal{H}=(V,E)$ be a connected hypergraph. Let $f$ be a hypertree of $\mathcal{H}$.  For a subset $E'\subset E$, we say that $E'$ is  tight at $f$ if  $\sum\limits_{e\in E'}f(e)= \mu(E')$.
\end{definition}
It is clear that if $E'$ is  tight at $f$ and $\sum\limits_{e\in E'}f(e)= \sum\limits_{e\in E'}g(e)$ for another hypertree $g$, then $E'$ is  tight at $g$. Moreover, the following theorem holds immediately from Theorem 44.2 in \cite{Schrijver}.
\begin{theorem}\label{thm 2}
	Let $\mathcal{H}=(V,E)$ be a connected hypergraph. Let $f$ be a hypertree of $\mathcal{H}$.  If the subsets $A\subset E$ and $B\subset E$ are both tight at $f$, then $A\cap B$ and $A\cup B$ are both tight at $f$.
\end{theorem}
In \cite{Kalman1}, the following result was presented by using Theorem \ref{thm 2}.

\begin{lemma}[\cite{Kalman1}]\label{lem 9}
Let $G=(V\cup E, \varepsilon)$ be a connected bipartite graph. Let $f$ be a hypertree of $G$. Then for any non-empty subset $E' \subset E$,  if $E'$ is not tight at $f$, then $f$ is a hypertree  so that a transfer of valence is possible from  some element of $E\setminus E'$ to  some element of $E'$.
\end{lemma}
Next, based on Lemmas  \ref{lem 3} and \ref{lem 9}, we give a necessary and sufficient condition that a hyperedge can transfer valence to another  hyperedge with respect to a hypertree.

\begin{lemma}\label{lem 6}
Let $\mathcal{H}=(V,E)$ be a connected hypergraph, and let $f$ be a hypertree of $\mathcal{H}$. Let $e$ and $e'$ be two distinct  hyperedges of $\mathcal{H}$. $e$ can transfer valence to $e'$  with respect to  $f$ if and only if  $f(e)\neq 0$ and every subset $E'\subset  E$ which contains $e'$ and does not contain $e$, is not tight at $f$.
\end{lemma}
\begin{proof}
The necessity is obvious. For sufficiency, let us take $E_{1}=\{e'\}$. Since $E_{1}$ is not tight at $f$ some element of $E\setminus E_{1}$ can transfer valence to $e'$ for $f$ by Lemma \ref{lem 9}. Let $U_{1}$ be the set consisting of  all elements of $E\setminus E_{1}$ that can transfer valence to $e'$ for $f$. If $e\in U_{1}$, then the conclusion is true. If $e\notin U_{1}$, then we take $E_{2}=U_{1}\cup E_{1}$. Note that $E_{2}$ is not tight at $f$. Then some element of $E\setminus E_{2}$ can transfer  valence to some element of $E_{2}$ for $f$ by Lemma \ref{lem 9}. Let $U_{2}$ be the set consisting of  all elements of $E\setminus E_{2}$ that can transfer  valence to some element of $E_{2}$ for $f$. It is obvious that $E_{1}$ is a proper subset of $E_{2}$. Moreover, all elements of $E_{3}=U_{2}\cup E_{2}$ (except for $e'$) can transfer  valence to $e'$ for $f$ by Lemma \ref{lem 3}, and  $E_{2}$ is a proper subset of $E_{3}$. Continue the above process, we will eventually obtain that $e$ can transfer valence to $e'$ for $f$.
\end{proof}
In the end of this section, let $\mathcal{H}=(V,E)$ be a connected hypergraph, and let $e_{1}$ and $e_{2}$ be two distinct hyperedges of $\mathcal{H}$. Let $f_{1}$ and $f_{2}$ be two hypertrees of $\mathcal{H}$ with $f_{1}(e_{1})<f_{2}(e_{1})$ and $f_{1}(e)=f_{2}(e)$ for any  $e\in E\setminus\{e_{1},e_{2}\}$. Firstly,
we present one known result which is proved by technique of the hypertree polytope in \cite{Kalman1}. However, we prove it again by using  basic concept.
\begin{lemma}[\cite{Kalman1}]\label{lem 5}
(1) If $f_{1}$ is a hypertree such that valence can  be transferred from $e_{2}$ to $e$, then $f_{2}$ is a hypertree such that valence can  be transferred from $e_{1}$ to $e$.

(2)  If $f_{1}$ is a hypertree such that valence can  be transferred from $e$ to $e_{1}$, then $f_{2}$ is a hypertree such that valence can  be transferred from $e$ to $e_{2}$.
\end{lemma}
\begin{proof}
(1) It is clear that $f_{2}(e_{1})\neq 0$ due to $f_{1}(e_{1})<f_{2}(e_{1})$. We claim that any subset $E_{1}\subset  E$  which contains $e$ and does not contain $e_{1}$, is not tight at $f_{2}$. Assume that $e_{2}\notin E_{1}$. By Lemma \ref{lem 6}, $E_{1}$ is not tight at $f_{1}$ as $e_{2}$ can transfer valence to $e$  with respect to  $f_{1}$. In this case, $\sum\limits_{e\in E_{1}}f_{1}(e)= \sum\limits_{e\in E_{1}}f_{2}(e)$ by construction of hypertrees $f_{1}$ and $f_{2}$. Then $E_{1}$ is not tight at $f_{2}$.
Assume that $e_{2}\in E_{1}$.  Then  $\sum\limits_{e\in E_{1}}f_{2}(e)=\sum\limits_{e\in E_{1}\setminus e_{2}}f_{2}(e)+f_{2}(e_{2})= \sum\limits_{e\in E_{1}\setminus e_{2}}f_{1}(e)+f_{2}(e_{2})<\sum\limits_{e\in E_{1}\setminus e_{2}}f_{1}(e)+f_{1}(e_{2})=\sum\limits_{e\in E_{1}}f_{1}(e)\leq \mu(E_{1})$. It implies that $E_{1}$ is not tight at $f_{2}$. Thus, by Lemma \ref{lem 6}, $e_{1}$ can transfer valence to $e$  with respect to  $f_{2}$.

(2) Since $e$ can transfer valence to $e_1$ with respect to $f_1$, $f_2(e)=f_1(e)\neq 0$. We further claim that any subset $E_{2}\subset  E$ which contains $e_{2}$ and does not contain $e$, is not tight at $f_{2}$. (i) Assume that $e_{1}\in E_{2}$. Then $E_{2}$ is not tight at $f_{1}$ as $e_{1}$ can be transferred valence from $e$  with respect to  $f_{1}$. In this case,  $\sum\limits_{e\in E_{2}}f_{1}(e)= \sum\limits_{e\in E_{2}}f_{2}(e)$ by construction of hypertrees $f_{1}$ and $f_{2}$. Then $E_{2}$ is not tight at $f_{2}$.
(ii) Assume that $e_{1}\notin E_{2}$.  Then  $\sum\limits_{e\in E_{2}}f_{2}(e)=\sum\limits_{e\in E_{2}\setminus e_{2}}f_{2}(e)+f_{2}(e_{2})= \sum\limits_{e\in E_{2}\setminus e_{2}}f_{1}(e)+f_{2}(e_{2})<\sum\limits_{e\in E_{2}\setminus e_{2}}f_{1}(e)+f_{1}(e_{2})=\sum\limits_{e\in E_{2}}f_{1}(e)\leq \mu(E_{2})$. It implies that $E_{2}$ is not tight at $f_{2}$. Thus, by Lemma \ref{lem 6}, $e_{2}$ can be transferred valence from $e$  with respect to  $f_{2}$.
\end{proof}
Next for two distinct hyperedges $e,e'\in E\backslash \{e_{1},e_{2}\}$, we prove the following result.

\begin{lemma}\label{lem 10}
(1) If  $e$ and $e_{2}$ can not transfer  valence to $e'$ with respect to  $f_{1}$, then $e$ and $e_{2}$ can not transfer  valence to $e'$ with respect to  $f_{2}$.

(2) If  $e$ can  transfer  valence to neither $e_{1}$ nor $e'$ with respect to  $f_{1}$, then  $e$ can  transfer  valence to neither $e_{1}$ nor $e'$ with respect to  $f_{2}$.
\end{lemma}
\begin{proof}
If  $e$ can not transfer  valence to $e'$ with respect to  $f_{1}$, then by lemma \ref{lem 6}, we have that (i) $f_{1}(e)=0$ or (ii) there is a subset $E_{1}\subset E$ which contains $e'$ and does not contain $e$, is  tight at $f_{1}$.

(1) If (i) holds, then $e$ can not transfer  valence to  $e'$ with respect to  $f_{2}$ as $f_2(e)=f_1(e)=0$.   Since $e_{2}$ can not transfer  valence to $e'$ with respect to  $f_{1}$,  there is a subset $E_{2}\subset E$ which contains $e'$ and does not contain $e_{2}$, such that it is  tight at $f_{1}$  (note that $f_{1}(e_{2})\neq 0$). It is clear that $\sum\limits_{e\in E_{2}}f_{1}(e)\leq \sum\limits_{e\in E_{2}}f_{2}(e)$ as $e_2\notin E_{2}$.  Then  $E_{2}$ is also tight at $f_{2}$, that is, $e_{2}$  cannot transfer valence to $e'$ in $f_{2}$ by lemma \ref{lem 6}. If (ii) holds,   we take $E'=E_{1}\cap E_{2}$. Then, by Theorem \ref{thm 2}, the subset $E'\subset  E$ which contains  $e'$ and does not contain $e$ and $e_{2}$, is tight at $f_{1}$. Since $e_2\notin E'$, we have $\sum\limits_{e\in E'}f_{1}(e)\leq \sum\limits_{e\in E'}f_{2}(e)$. Thus $E'$ is also tight at $f_{2}$. This implies that $e_{2}$ and $e$ cannot transfer valence to $e'$ in $f_{2}$.

(2) If (i) holds, then $e$ can  transfer  valence to neither $e_{1}$ nor $e'$ with respect to  $f_{2}$ as $f_2(e)=f_1(e)=0$.  Assume (ii) holds. If $e$ can not transfer  valence to $e_{1}$ with respect to  $f_{1}$, then there is a subset $E_{3}\subset E$ which contains $e_{1}$ and does not contain $e$, such that it is  tight at $f_{1}$. Take $E''=E_{1}\cup E_{3}$. Then the subset $E''\subset  E$, which contains  $e'$ and $e_{1}$ and does not contain $e$, is tight at $f_{1}$ by Theorem \ref{thm 2}. Since $e_1\in E''$, we have $\sum\limits_{e\in E''}f_{1}(e)\leq \sum\limits_{e\in E''}f_{2}(e)$. Thus $E''$ is also tight at $f_{2}$. This implies that $e$ can  transfer  valence to neither $e_{1}$ nor $e'$ in $f_{2}$.
\end{proof}
Given an order on $E$, for any hyperedge $e>e_{1}$ and the hypertree  $f_{1}$, if $e$  is internally  active, then $e$ can not transfer  valence to $e_{1}$. By Lemmas \ref{lem 10} (1) and \ref{lem 5} (1), we know that $e$  is internally  active with respect to  $f_{2}$. Similarly, for any hyperedge $e'>e_{2}$ and  the hypertree $f_{1}$, if $e'$  is externally  active, then $e'$ can not be transferred  valence from $e_{2}$. By Lemmas \ref{lem 10} (2) and \ref{lem 5} (2), we know that $e$  is externally  active with respect to  $f_{2}$. We have the similar result if for the hypertree  $f_{2}$, $e>e_{2}$  is internally  active and $e'>e_{1}$  is externally  active. Thus the following conclusion holds.
\begin{lemma}\label{lem 4}
Given an order on $E$,  for any hyperedge $e>\max\{e_{1},e_{2}\}$, $e$  is internally (externally, resp.) active with respect to  $f_{1}$ if and only if it  is internally (externally, resp.) active with respect to  $f_{2}$.
\end{lemma}
\section{Proof of Theorem 1}
\noindent

It suffices to prove that the interior and exterior polynomials are equal under the following distinct orders $O$ and $O'$ on $E$, respectively, that is, $I_{\mathcal{H},O}(x)=I_{\mathcal{H},O'}(x)$ and $X_{\mathcal{H},O}(y)=X_{\mathcal{H},O'}(y)$, where

$O:e_{1}<e_{2}<\cdots<e_{h-1}<e_{h}<e_{h+1}<e_{h+2}<\cdots<e_{|E|}$, and

$O':e_{1}<e_{2}<\cdots<e_{h-1}<e_{h+1}<e_{h}<e_{h+2}<\cdots<e_{|E|}$.

For convenience, let $E_{1}=\{e_{1},e_{2},\ldots,e_{h-1}\}$, $E_{2}=E\setminus\{e_{h},e_{h+1}\}$, $E_{3}=E_{1}\cup \{e_{h}\}$ and $E_{4}=E_{1}\cup \{e_{h+1}\}$. Firstly, we have the following two obvious facts.

\textbf{Fact 1.} For any hypertree of $\mathcal{H}$ and any hyperedge $e_{i}\in E_{2}$, $e_{i}$ is internally (externally, resp.) active in $O$ if and only if $e_{i}$ is internally (externally, resp.) active in $O'$.

\textbf{Fact 2.} If $E'\subset E''$, and there is a hyperedge $e'\in E'$ such that $e'$ can be transferred valence from $e$ ($e'$ can  transfer valence to $e$,  resp.), then there is a hyperedge $e'\in E''$ such that $e'$ can be transferred valence from $e$ ($e'$ can  transfer valence to $e$,  resp.).

Let $f: (f(e_{1}),f(e_{2}),\cdots,f(e_{h-1}),f(e_{h}),f(e_{h+1}),f(e_{h+2}),\cdots,f(e_{|E|})$) be a hypertree of $\mathcal{H}$. When we discuss the activity of a hyperedge in an ordering of hyperedges with respect to the hypertree $f$, if in the context the hypertree $f$ is clear,  for simplicity, sometimes we will not state it explicitly.

By Fact 2, we have the following two claims immediately.

\textbf{Claim 1.} If $e_{h}$ is internally (externally, resp.) inactive  in $O$, then $e_{h}$ is internally (externally, resp.) inactive  in $O'$.

\textbf{Claim 2.} If $e_{h+1}$ is internally (externally, resp.) active  in $O$, then $e_{h+1}$ is internally (externally, resp.) active  in $O'$.

To deduce the relations of $\overline{\iota}_{O}(f)$ and $\overline{\iota}_{O'}(f)$, and $\overline{\epsilon}_{O}(f)$ and $\overline{\epsilon}_{O'}(f)$, we give the following claims by the existence of hypertrees $f_{1}$ and $f_{2}$, where

$f_{1}$: $(f(e_{1}),f(e_{2}),\cdots,f(e_{h-1}),f(e_{h})+1,f(e_{h+1})-1,f(e_{h+2}),\cdots,f(e_{|E|}))$, and

$f_{2}$: $(f(e_{1}),f(e_{2}),\cdots,f(e_{h-1}),f(e_{h})-1,f(e_{h+1})+1,f(e_{h+2}),\cdots,f(e_{|E|}))$.

\textbf{Claim 3.} If $f_{1}$ is a hypertree and $e_{h}$ is internally inactive (externally active, resp.)  in $O$, then $e_{h+1}$ is internally inactive (externally active, resp.)  in $O'$.

\emph{Proof of Claim 3.} We first consider that $e_{h}$ is internally inactive in $O$. Then there is a hyperedge $e_{i}\in E_{1}$ such that $e_{i}$ can be transferred valence from $e_{h}$.   It is clear that $e_{h+1}$ can transfer valence to $e_{h}$  due to the existence of hypertree $f_{1}$. By Lemma \ref{lem 3}, $e_{h+1}$ can transfer valence to $e_{i}\in E_{1}$. Thus, $e_{h+1}$ is internally inactive in $O'$.

Now we consider that $e_{h}$ is externally active  in $O$. Assume that the opposite is true, that is, $e_{h+1}$ is externally inactive  in $O'$. Then there is a hyperedge $e_{i}\in E_{1}$ such that $e_{i}$ can transfer valence to $e_{h+1}$.  $e_{h+1}$ can transfer valence to $e_{h}$ due to the existence of hypertree $f_{1}$. By Lemma \ref{lem 3}, $e_{h}$ can be transferred valence from $e_{i}\in E_{1}$, which contradicts the fact that $e_{h}$ is externally active  in $O$. Hence, $e_{h+1}$ is externally active  in $O'$.
$\qed$

\textbf{Claim 4.} Assume that $f_{1}$ is not a hypertree.

(i) If $e_{h+1}$ is internally inactive  in $O$,  then $e_{h+1}$ is internally inactive  in $O'$.

(ii) If $e_{h}$ is externally active  in $O$,  then $e_{h}$ is externally active  in $O'$.

\emph{Proof of Claim 4.} (i) If $e_{h+1}$ is internally inactive  in $O$, then there is a hyperedge $e_{i}\in E_{3}$ such that $e_{i}$ can be transferred valence from $e_{h+1}$.  $e_{h+1}$ can not transfer valence to $e_{h}$  as $f_{1}$ is not a  hypertree. Then  $e_{i}\in E_{1}$. Hence, $e_{h+1}$ is internally inactive in $O'$.

(ii) If $e_{h}$ is externally active  in $O$, then for any hyperedge $e_{i}\in E_{1}$,  $e_{i}$ can not transfer valence to $e_{h}$. We know that $e_{h+1}$ can not transfer valence to $e_{h}$ as $f_{1}$ is not a hypertree. It implies for any hyperedge $e_{j}\in E_{4}$,  $e_{j}$ can not transfer valence to $e_{h}$. Thus, $e_{h}$ is externally active  in $O'$.$\qed$

\textbf{Claim 5.}
Assume that $f_{2}$ is a hypertree.

(i) If $e_{h}$ is internally active  in $O$,  then $e_{h+1}$ is internally active  in $O'$.

(ii) If $e_{h}$ is externally inactive  in $O$,  then $e_{h+1}$ is externally inactive  in $O'$.

\emph{Proof of Claim 5.}   (i) Assume that the opposite is true, that is, $e_{h+1}$ is internally inactive  in $O'$. Then there is a hyperedge $e_{i}\in E_{1}$ such that $e_{i}$ can be transferred valence from $e_{h+1}$.  $e_{h}$ can transfer valence to $e_{h+1}$  due to the existence of the hypertree $f_{2}$. By Lemma \ref{lem 3}, $e_{h}$ can transfer valence to $e_{i}\in E_{1}$, which contradicts the fact that $e_{h}$ is internally active  in $O$. Hence, $e_{h+1}$ is internally active  in $O'$.

(ii) If $e_{h}$ is externally inactive  in $O$, then there is a hyperedge $e_{i}\in E_{1}$ such that $e_{i}$ can transfer valence to $e_{h}$. $e_{h}$ can transfer valence to $e_{h+1}$  due to the existence of the hypertree $f_{2}$. By Lemma \ref{lem 3}, $e_{i}\in E_{1}$ can transfer valence to $e_{h+1}$. Thus, $e_{h+1}$ is externally inactive  in $O'$. $\qed$

\textbf{Claim 6.} Assume that $f_{2}$ is not a hypertree.

(i) If $e_{h}$ is internally active  in $O$,  then $e_{h}$ is internally active  in $O'$.

(ii) If $e_{h+1}$ is externally inactive  in $O$,  then $e_{h+1}$ is externally inactive  in $O'$.

\emph{Proof of Claim 6.}  (i) If $e_{h}$ is internally active  in $O$, then for any hyperedge $e_{i}\in E_{1}$, $e_{i}$ can not be transferred valence from $e_{h}$. $e_{h+1}$ can not be transferred valence from $e_{h}$ as $f_{2}$ is not a hypertree. Thus
for any hyperedge $e_{j}\in E_{4}$, $e_{j}$ can not be transferred valence from $e_{h}$, that is, $e_{h}$ is internally active  in $O'$.

(ii) If $e_{h+1}$ is externally inactive  in $O$, then there is a hyperedge $e_{i}\in E_{3}$ such that $e_{i}$ can transfer valence to $e_{h+1}$. $e_{h}$ can not transfer valence to $e_{h+1}$ as $f_{2}$ is not a hypertree. Then  $e_{i}\in E_{1}$. Hence, $e_{h+1}$ is externally inactive  in $O'$.
$\qed$

Next, we divide into three cases to discuss the relations of $\overline{\iota}_{O}(f)$ and $\overline{\iota}_{O'}(f)$, and $\overline{\epsilon}_{O}(f)$ and $\overline{\epsilon}_{O'}(f)$ by using previous claims.

\textbf{Case 1.} $f_{1}$ is not a hypertree and $f_{2}$ is not a hypertree.

We know from Claims 1, 2, 4 and 6 that $e_{h}$ ($e_{h+1}$, resp.) is internally active in $O$ if and only if $e_{h}$ ($e_{h+1}$, resp.) is internally active in $O'$. The similar conclusion holds when $e_{h}$ and $e_{h+1}$  are externally active. Combining with Fact 1, we have $\overline{\iota}_{O}(f)=\overline{\iota}_{O'}(f)$ and $\overline{\epsilon}_{O}(f)=\overline{\epsilon}_{O'}(f)$.

\textbf{Case 2.} $f_{1}$ and $f_{2}$ are both hypertrees.

Clearly, $e_{h+1}$ ($e_{h}$, resp.) is internally and externally inactive in $O$ ($O'$, resp.). Moreover, we know from Claims 3 and 5 that $e_{h}$  is internally (externally, resp.) active in $O$ if and only if $e_{h+1}$  is internally (externally, resp.) active in $O'$.
Combining with Fact 1, we have $\overline{\iota}_{O}(f)=\overline{\iota}_{O'}(f)$ and
  $\overline{\epsilon}_{O}(f)=\overline{\epsilon}_{O'}(f)$.

\textbf{Case 3.} One of $f_{1}$ and $f_{2}$ is a hypertree, and the other is not a hypertree.

In this case, we consider a subset $\mathcal{F}_f$ of $B_{\mathcal{H}}$ associated to a hypertree $f$ of $\mathcal{H}$, where
\begin{eqnarray*}
\mathcal{F}_f&=&\left\{g:E\rightarrow N\bigg|\sum\limits_{e\in E}g(e)=|V|-1 \text{ and }  g(e)=f(e) \text{ for any hyperedge}\ e\in E_{2}\right\}\\
&=&\{g_{i}:i=1,2,\cdots,l\}
\end{eqnarray*}
 and $g_{1}(e_{h})<g_{2}(e_{h})<\cdots<g_{l}(e_{h})$. It is clear that $g_{i}(e_{h})=g_{i+1}(e_{h})-1$ for each $i=1,2,\cdots,l-1$ and $l\geq2$. Moreover, $f=g_{1}$ or $f=g_{l}$. Without loss of generality we assume that $f_{1}$ is a hypertree and $f_{2}$ is not a hypertree, that is, $f=g_{1}$. Let $f^{\ast}=g_{l}$. We prove that either $\overline{\iota}_{O}(f)=\overline{\iota}_{O'}(f)$ and $\overline{\iota}_{O}(f^{\ast})=\overline{\iota}_{O'}(f^{\ast})$, or $\overline{\iota}_{O}(f)=\overline{\iota}_{O'}(f^{\ast})$ and $\overline{\iota}_{O}(f^{\ast})=\overline{\iota}_{O'}(f)$, and either $\overline{\epsilon}_{O}(f)=\overline{\epsilon}_{O'}(f)$ and $\overline{\epsilon}_{O}(f^{\ast})=\overline{\epsilon}_{O'}(f^{\ast})$, or $\overline{\epsilon}_{O}(f)=\overline{\epsilon}_{O'}(f^{\ast})$ and $\overline{\epsilon}_{O}(f^{\ast})=\overline{\epsilon}_{O'}(f)$ as follows.

It is clear that with respect to  $f$,  $e_{h+1}$ can transfer valence to $e_{h}$, that is, $e_{h+1}$ is internally inactive  in $O$ and $e_{h}$ is externally inactive  in $O'$.
By Claims 1 and 6 (i), we know that with respect to  $f$, $e_{h}$  is internally active in $O$ if and only if $e_{h}$  is internally active in $O'$. Moreover, by Claims 2 and 6 (ii), with respect to $f$, $e_{h+1}$  is externally active in $O$ if and only if $e_{h+1}$ is externally active in $O'$.  Combining with Fact 1, we have that with respect to $f$, for any $e\in E_{3}$, $e$  is internally active in $O$ if and only if $e$ is internally active in $O'$, and for any $e'\in E_{4}$, $e'$  is externally active in $O$ if and only if $e'$ is externally active in $O'$.

It is clear that with respect to  $f^{\ast}$, $e_{h}$ can  transfer valence to $e_{h+1}$, that is, $e_{h}$ is internally inactive  in $O'$ and $e_{h+1}$ is externally inactive  in $O$. Since $e_{h+1}$ can not transfer valence to $e_{h}$ with respect to  $f^{\ast}$, by Claims 2 and  4 (i), we know that with respect to  $f^{\ast}$, $e_{h+1}$  is internally active in $O$ if and only if $e_{h+1}$  is internally active in $O'$. Moreover, by Claims 1 and 4 (ii), with respect to $f^{\ast}$, $e_{h}$  is externally active in $O$ if and only if $e_{h}$ is externally active in $O'$. Combining with Fact 1, we have that  with respect to $f^{\ast}$, for any $e\in E_{4}$, $e$  is internally active in $O$ if and only if $e$ is internally active in $O'$, and for any $e'\in E_{3}$, $e'$  is externally active in $O$ if and only if $e'$ is externally active in $O'$.

We need to consider the internally activity of $e_{h+1}$ in $O'$ and externally activity of $e_{h}$ in $O$ with respect to  $f$ as follows.

Firstly, the internally activity of $e_{h+1}$ with respect to  $f$ in $O'$ is considered.

If $e_{h+1}$  is internally inactive with respect to  $f$ in $O'$, then $\overline{\iota}_{O}(f)=\overline{\iota}_{O'}(f)$, and there is a hyperedge $e_{i}\in E_{1}$ such that $e_{i}$ can be transferred valence from $e_{h+1}$ with respect to  $f$. By  Lemma \ref{lem 5} (1), we know that $e_{i}$ can be transferred valence from $e_{h}$ with respect to  $f^{\ast}$, that is, $e_{h}$  is internally inactive in $O$ with respect to  $f^{\ast}$. Thus, $\overline{\iota}_{O}(f^{\ast})=\overline{\iota}_{O'}(f^{\ast})$.

If $e_{h+1}$  is internally active with respect to $f$ in $O'$, then  $\overline{\iota}_{O}(f)=\overline{\iota}_{O'}(f)+1$. To prove  $\overline{\iota}_{O}(f)=\overline{\iota}_{O'}(f^{\ast})$ and   $\overline{\iota}_{O}(f^{\ast})=\overline{\iota}_{O'}(f)$, we need the following two claims.

\textbf{Claim 7.} With respect to  $f$ and $f^{\ast}$,  $e_{h}$ ($e_{h+1}$, resp.) is internally active in $O$ ($O'$, resp.).

\emph{Proof of Claim 7.} Since $e_{h+1}$  is internally active in $O'$ with respect to  $f$, for any hyperedge $e_{i}\in E_{1}$, $e_{i}$ can not be transferred valence from   $e_{h+1}$ with respect to  $f$. It follows from Lemma \ref{lem 6} that there is a  subset $E'\subset  E$ which contains $e_{i}$ and does not contain $e_{h+1}$, such that it is  tight at $f$ (note that $f(e_{h+1})\neq 0$).
It is clear that $E'$ does not contain $e_{h}$. Otherwise, $\mu(E')=\sum\limits_{e\in E'}f(e)<\sum\limits_{e\in E'}f^{\ast}(e)$, a contradiction. Then $\mu(E')=\sum\limits_{e\in E'}f(e)=\sum\limits_{e\in E'}f^{\ast}(e)$, that is, $E'$ is  also tight at $f^{\ast}$.  We have that with respect to $f$ and $f^{\ast}$ and for any hyperedge $e_{i}\in E_{1}$, $e_{i}$ can be transferred valence from neither $e_{h}$ nor $e_{h+1}$. Thus, the claim holds. $\qed$

\textbf{Claim 8.}  For the order $O$ and any hyperedge $e_{i}\in E_{2}$, $e_{i}$ is internally active with respect to $f$ if and only if $e_{i}$ is internally active with respect to  $f^{\ast}$.

\emph{Proof of Claim 8.}
We firstly consider that $e_{i}\in E_{1}$. By Claim 7, we have that $e_{h+1}$ can  not transfer valence to  $e_{i}$ with respect to  $f$ and $e_{h}$ can not  transfer valence to  $e_{i}$ with respect to  $f^{\ast}$. Since for any hyperedge $e_{j}<e_{i}$, $e_{j}\neq e_{h+1},e_{h}$, the conclusion holds by Lemma \ref{lem 10} (1).

By Lemma \ref{lem 4}, the claim holds for $e_{i}\in E_{2}\setminus E_{1}$. Thus, this completes the proof of the claim.
$\qed$

Thus, combining Fact 1 with Claims 7 and 8 (summarized in Tables 1 and 2), we have $\overline{\iota}_{O}(f)=\overline{\iota}_{O'}(f^{\ast})$ and  $\overline{\iota}_{O'}(f)=\overline{\iota}_{O}(f^{\ast})$.

\begin{table}[htp]
\begin{center}
\caption{Internal activity for $e_h$ and $e_{h+1}$,  $IA$ and $II$ denote internally active and internally inactive, respectively.}
 \begin{tabular}{|c |c|c|c|c|}
  \hline
  & $(f,O)$& $(f^{\ast},O')$& $(f,O')$& $(f^{\ast},O)$ \\
 \hline
 $e_{h}$&IA (Claim 7)& II& IA& IA (Claim 7)\\
 \hline
 $e_{h+1}$&II& IA (Claim 7)& IA (Claim 7)& IA\\
 \hline
\end{tabular}
\end{center}
\end{table}

\begin{table}[htp]
\begin{center}
\caption{Internal activity for each hyperedge in $E_2$.}
 \begin{tabular}{ c cc c c }

  $(f,O)$& $\stackrel{\text{\tiny Claim 8}}{\Large \Longleftrightarrow}$& $(f^{\ast},O)$ \\

 \tiny{Fact 1}\Large
 $\Updownarrow$& & \tiny{Fact 1}\Large
 $\Updownarrow$ \\

 $(f,O')$& & $(f^{\ast},O')$\\

\end{tabular}
\end{center}
\end{table}

Next, the externally activity of $e_{h}$ with respect to  $f$ in $O$ is considered.
If $e_{h}$  is externally inactive with respect to  $f$ in $O$, then $\overline{\epsilon}_{O}(f)=\overline{\epsilon}_{O'}(f)$, and  there is a hyperedge $e_{i}\in E_{1}$ such that $e_{i}$ can  transfer valence to $e_{h}$ with respect to  $f$. By  Lemma \ref{lem 5} (2), we know that $e_{i}$  can transfer valence to $e_{h+1}$ with respect to  $f^{\ast}$. This implies that $e_{h+1}$  is externally inactive in $O'$ with respect to  $f^{\ast}$.  Thus,  $\overline{\epsilon}_{O}(f^{\ast})=\overline{\epsilon}_{O'}(f^{\ast})$.

Assume that $e_{h}$  is externally active in $O$ with respect to  $f$. Then  $\overline{\epsilon}_{O}(f)=\overline{\epsilon}_{O'}(f)-1$. To prove  $\overline{\epsilon}_{O}(f)=\overline{\epsilon}_{O'}(f^{\ast})$ and   $\overline{\epsilon}_{O}(f^{\ast})=\overline{\epsilon}_{O'}(f)$, we need the following two claims.

\textbf{Claim $7'$.} With respect to $f$ and  $f^{\ast}$, $e_{h}$ ($e_{h+1}$, resp.) is externally active  in $O$ ($O'$, resp.).

\emph{Proof of Claim $7'$.} Since $e_{h}$  is externally active in $O$ with respect to  $f$, for any hyperedge $e_{i}\in E_{1}$, $e_{i}$ can not transfer valence to  $e_{h}$. It follows from Lemma \ref{lem 6} that either (i) $f(e_{i})=0$ or (ii) there is a  subset $E''\subset  E$ which contains $e_{h}$ and does not contain $e_{i}$, such that it is  tight at $f$. If $f(e_{i})=0$, then $e_{i}$ can transfer valence to neither $e_{h}$ nor $e_{h+1}$ with respect to $f$ and $f^{\ast}$ as $f^{\ast}(e_{i})=f(e_{i})=0$.
Assume that (ii) holds.
Then  $E''$  contains $e_{h+1}$. Otherwise, $\mu(E'')=\sum\limits_{e\in E''}f(e)<\sum\limits_{e\in E''}f^{\ast}(e)$, a contradiction. It implies that  $\mu(E'')=\sum\limits_{e\in E''}f(e)=\sum\limits_{e\in E''}f^{\ast}(e)$, that is, $E'$ is also tight at $f^{\ast}$.  We have that   with respect to $f$ and  $f^{\ast}$, and  any hyperedge $e_{i}\in E_{1}$ can transfer valence to neither $e_{h}$ nor $e_{h+1}$.  Thus, the claim  holds.$\qed$

\textbf{Claim $8'$.}  For the order $O$ and any hyperedge $e_{i}\in E_{2}$, $e_{i}$ is externally active  with respect to  $f$ if and only if $e_{i}$ is externally active with respect to  $f^{\ast}$.

\emph{Proof of Claim $8'$.}  We firstly consider that $e_{i}\in E_{1}$. By Claim $7'$, we have that for any hyperedge $e_{j}<e_{i}$ (note that $e_{j}\neq e_{h+1},e_{h}$),  $e_{j}$ can  not transfer valence to  $e_{h}$ with respect to  $f$ and $e_{j}$ can not  transfer valence to  $e_{h+1}$ with respect to  $f^{\ast}$. By Lemma \ref{lem 10} (2), the conclusion holds.

By Lemma \ref{lem 4}, the claim holds for $e_{i}\in E_{2}\setminus E_{1}$. Thus, this completes the proof of the claim.
$\qed$

\begin{table}[htp]
\begin{center}
\caption{External activity for $e_h$ and $e_{h+1}$,  $EA$ and $EI$ denote externally active and externally inactive, respectively.}
 \begin{tabular}{|c |c|c|c|c|}
  \hline
  & $(f,O)$& $(f^{\ast},O')$& $(f,O')$& $(f^{\ast},O)$ \\
 \hline
 $e_{h}$&EA (Claim $7'$)& EI& EA& EA (Claim $7'$)\\
 \hline
 $e_{h+1}$&EI& EA (Claim $7'$)& EA (Claim $7'$)& EA\\
 \hline
\end{tabular}
\end{center}
\end{table}

\begin{table}[htp]
\begin{center}
 \caption{External activity for each hyperedge in $E_2$.}
 \begin{tabular}{ c cc c c }

  $(f,O)$& $\stackrel{\text{\tiny Claim $8'$}}{\Large \Longleftrightarrow}$& $(f^{\ast},O)$ \\

 \tiny{Fact 1}\Large
 $\Updownarrow$& & \tiny{Fact 1}\Large
 $\Updownarrow$ \\

 $(f,O')$& & $(f^{\ast},O')$\\

\end{tabular}
\end{center}
\end{table}

Thus, combining Fact 1 with Claims $7'$ and $8'$ (summarized in Tables 3 and 4) , we have $\overline{\epsilon}_{O}(f)=\overline{\epsilon}_{O'}(f_{1})$ and $\overline{\epsilon}_{O'}(f)=\overline{\epsilon}_{O}(f_{1})$.

Finally, we consider $\sum\limits_{g\in \mathcal{F}} x^{\overline{\iota}(g)}$ and $\sum\limits_{g\in \mathcal{F}} y^{\overline{\epsilon}(g)}$ under two distinct orders $O$ and $O'$. If $l=1$, then by Case 1, $\sum\limits_{g\in \mathcal{F}} x^{\overline{\iota}(g)}$ is equal under two distinct orders $O$ and $O'$, and it is similar for $\sum\limits_{g\in \mathcal{F}} y^{\overline{\epsilon}(g)}$.
If $l=2$, then  by Case 3, $\sum\limits_{g\in \mathcal{F}} x^{\overline{\iota}(g)}$ is equal under two distinct orders $O$ and $O'$, and it is similar for $\sum\limits_{g\in \mathcal{F}} y^{\overline{\epsilon}(g)}$. Assume $l\geq 3$. Then by Case 2, $\overline{\iota}_{O}(g_{i})=\overline{\iota}_{O'}(g_{i})$ and $\overline{\epsilon}_{O}(g_{i})=\overline{\epsilon}_{O'}(g_{i})$ for each $i=2,3,\cdots,l-1$, and by Case 3, $x^{\overline{\iota}(g_{1})}+x^{\overline{\iota}(g_{l})}$ is equal under two distinct orders $O$ and $O'$, and it is similar for $y^{\overline{\epsilon}(g_{1})}+y^{\overline{\epsilon}(g_{l})}$. We have that $\sum\limits_{g\in \mathcal{F}} x^{\overline{\iota}(g)}$ is equal under two distinct orders $O$ and $O'$, and it is similar for $\sum\limits_{g\in \mathcal{F}} y^{\overline{\epsilon}(g)}$.

Thus, $I_{\mathcal{H},O}(x)=I_{\mathcal{H},O'}(x)$ and $X_{\mathcal{H},O}(y)=X_{\mathcal{H},O'}(y)$.
This completes the proof.
$\qed$
\section*{Acknowledgements}
\noindent
This work is supported by NSFC (No. 12171402) and the Fundamental Research
Funds for the Central Universities  (No. 20720190062).

\section*{References}
\bibliographystyle{model1b-num-names}
\bibliography{<your-bib-database>}

\end{document}